\documentclass[11pt]{article}
\usepackage[T1]{fontenc}
\usepackage[utf8]{inputenc}
\usepackage[english]{babel}
\usepackage{xcolor}
\usepackage{amsmath,amssymb,amsthm,mathtools}
\usepackage[a4paper,margin=2.6cm]{geometry}
\usepackage{tikz}
\usetikzlibrary{calc,intersections}
\usepackage{float}
\usepackage{hyperref}

\newtheorem{theorem}{Theorem}
\newtheorem{lemma}[theorem]{Lemma}
\newtheorem{proposition}[theorem]{Proposition}

\usepackage{caption}

\theoremstyle{definition}
\newtheorem{remark}[theorem]{Remark}
\newtheorem*{remark*}{Remark}

\addto\captionsenglish{%
}

\newcommand{\R}{\mathbb R}
\newcommand{\pos}{\operatorname{pos}}

\title{The five distance theorem for an arbitrary norm}
\author{Nikita A. Mironov and Oleg R. Musin}
\date{}

\begin{document}
\maketitle

\begin{abstract}
The three gap theorem states that the points of the Kronecker sequence $\alpha,2\alpha,\ldots,N\alpha$, considered modulo one, divide the circle into intervals of at most three distinct lengths. In a two-dimensional nearest-neighbour
analogue, Haynes and Marklof proved that the Kronecker sequence modulo an
arbitrary unimodular lattice determines at most five distinct
nearest-neighbour distances in the Euclidean norm, and that this bound is
sharp. Dettmann subsequently constructed examples attaining five distinct
distances for every $\ell_p$-norm, $1\leq p\leq\infty$. We prove the
corresponding upper bound for every norm on $\mathbb R^2$:
for every full-rank lattice $L$, every
$\alpha\in\mathbb R^2$, and every $N\in\mathbb N$, the number
of distinct nearest-neighbour distances is at most five. For every norm, this bound is attained for some $L$ and
$\alpha$. For strictly convex norms, the proof extends the lattice-theoretic argument of Haynes and Marklof by replacing the Euclidean angular estimates with a cone lemma based on a proper Brass angular measure. The result for arbitrary norms is then obtained by a strictly convex perturbation and a limiting argument.
\end{abstract}

\section{Introduction}
A \emph{Kronecker sequence} on the $d$-dimensional torus
$\mathbb{T}^d=\mathbb{R}^d/L$, where $L\subset\mathbb{R}^d$ is a
full-rank lattice, is the sequence
$n\alpha+L$, $n=1,2,\ldots$,
generated by a vector $\alpha\in\mathbb{R}^d$. In dimension
one, with $L=\mathbb{Z}$ and $\mathbb{T}^1$ identified with $[0,1)$, this
takes the familiar form $x_n=\{n\alpha\}$, $n=1,2,\ldots$,
where $\{x\}$ denotes the fractional part of $x\in\mathbb{R}$. Weyl's
equidistribution theorem states that, for every irrational $\alpha$, the
sequence $(\{n\alpha\})_{n\geq 1}$ is uniformly distributed in $[0,1)$
and is therefore dense~\cite{Weyl}. The three-gap theorem asserts that,
for every $\alpha\in\mathbb{R}$ and every $N\geq 1$, the distinct points
among $\{\alpha\},\{2\alpha\},\ldots,\{N\alpha\}$ partition the circle $\mathbb{R}/\mathbb{Z}$ into intervals of at most
three distinct lengths.

A natural higher-dimensional analogue of the gaps between consecutive
points is provided by nearest-neighbour distances. Let
$\|\cdot\|$ be a norm on $\mathbb{R}^d$. Throughout, a lattice in $\mathbb R^d$ is assumed to have full rank. Following Haynes and
Marklof~\cite{HM}, we define
nearest-neighbour distances in the periodic lift of the finite Kronecker
sequence. In particular, a different lattice lift of the same torus point is
allowed, while the zero displacement vector is excluded. For $1\leq n\leq N$,
set
\begin{equation}\label{eq:delta-definition}
\delta_{n,N}^{\|\cdot\|}
=
\min\bigl\{
  \|n\alpha-m\alpha-\ell\|:
  1\leq m\leq N,\ \ell\in L,\
  n\alpha-m\alpha-\ell\ne0
\bigr\}.
\end{equation}

We then let
\[
    g_N^{\|\cdot\|}(\alpha,L)
       =
       \#\bigl\{
          \delta_{n,N}^{\|\cdot\|}:1\leq n\leq N
       \bigr\}
\]
denote the number of distinct nearest-neighbour distances. 

Biringer and Schmidt studied nearest-neighbour distances in finite
orbits of isometries and showed, in particular, that on a flat
$d$-dimensional torus their number is bounded above by
$3^d+1$~\cite{BiringerSchmidt}. For the standard two-dimensional torus, the bound
$g_N^{\ell_2}(\alpha,\mathbb Z^2)\leq5$
follows by combining Chevallier's correspondence between
nearest-neighbour distances and best-approximation denominators with
Romanov's four-step recurrence; see
\cite{Chevallier,Romanov,Shutov}.
Haynes and Marklof subsequently established this bound for every
unimodular lattice $L$ and showed that five distinct distances can
occur~\cite{HM}. Their result extends immediately to every full-rank lattice, since simultaneously rescaling $L$ and $\alpha$ multiplies all nearest-neighbour distances by the same positive constant.
They further conjectured that the corresponding sharp upper bound in dimension three is nine. This conjecture was recently confirmed by
Shulga, who also proved the estimate
$g_N^{\|\cdot\|}(\alpha,L)\leq2^d+1$
for every inner-product norm $\|\cdot\|$ on $\mathbb R^d$ and every
full-rank lattice $L\subset\mathbb R^d$~\cite{Shulga}. For the maximum norm, Chevallier proved the corresponding
estimate on the standard torus~\cite{Chevallier}:
$g_N^{\ell_\infty}(\alpha,\mathbb Z^2)\leq 5$.
Haynes and Ramirez later proved that, for every unimodular lattice $L$,
$g_N^{\ell_\infty}(\alpha,L)\leq 5$,
and that this bound is sharp~\cite{HaynesRamirez}.

Dettmann subsequently constructed examples on the standard torus
$\mathbb R^2/\mathbb Z^2$ with five distinct nearest-neighbour distances
for every $\ell_p$-norm, $1\leq p\leq\infty$, and conjectured that five
is the universal maximum in dimension two~\cite{Dettmann}. We prove the
conjectured upper bound in greater generality, for every norm and every
full-rank lattice. We also establish sharpness for every norm by
constructing an example on a suitable full-rank lattice; unlike
Dettmann's fixed standard lattice, the lattice in our construction is
adapted to the norm. For comparison, Shutov, working on the
standard torus under an irrationality assumption and with a different
nearest-neighbour convention, proved an upper bound of $K+1$, where $K$ is
the contact number of the unit ball. For the $\ell_p$-norm in dimension two,
$1<p<\infty$, this bound is $7$~\cite{Shutov}.

Our main results establish the conjectured upper bound for every norm
on $\mathbb R^2$ and show that it is sharp in this full generality.

\begin{theorem}\label{thm:main}
Let $L\subset\mathbb{R}^2$ be a full-rank lattice, let
$\alpha\in\mathbb{R}^2$, let $N\in\mathbb{N}$, and let
$\|\cdot\|$ be an arbitrary norm on $\mathbb{R}^2$. Then
\[
    g_N^{\|\cdot\|}(\alpha,L)\leq5.
\]
\end{theorem}

The bound in Theorem~\ref{thm:main} cannot be improved for any norm.

\begin{theorem}\label{thm:sharpness}
For every norm $\|\cdot\|$ on $\mathbb R^2$, there exist a full-rank
lattice $L\subset\mathbb R^2$ and a vector
$\alpha\in\mathbb R^2$ such that
\[
    g_{17}^{\|\cdot\|}(\alpha,L)=5.
\]
\end{theorem}

The proof of Theorem~\ref{thm:main} retains the lattice-theoretic
reduction and the norm-independent combinatorial part of the argument
of Haynes and Marklof. For strictly convex norms, their specifically
Euclidean angular argument is replaced by a cone lemma for normed
planes, proved using a proper Brass angular measure. The strict
convexity assumption is then removed by perturbing an arbitrary norm
with a small Euclidean component and passing to the limit.

To prove Theorem~\ref{thm:sharpness}, we adapt the rational configuration
underlying Dettmann's construction to the geometry of an arbitrary
unit ball. A good Brass angular
measure is used to select suitable smooth unit directions. A small perturbation of the
generator then separates four critical lattice minima and produces
exactly five distinct nearest-neighbour distances.

\section{Brass measure and auxiliary lemmas}
\label{sec:brass-measure}

Recall that a normed plane is strictly convex if its
unit circle contains no nontrivial line segment.
Let \((\mathbb V,\|\cdot\|)\) be such a plane, set
\(B=\{x\in\mathbb V:\|x\|\leq1\}\) and \(S=\partial B\), and write
\(O\) for their common centre. For $a,b\in S$ with $b\ne-a$, let
$\widehat{aOb}$ denote the minor arc of $S$ from $a$ to $b$.
When $b=-a$, either semicircle with endpoints $a$ and $b$ may be chosen. An \emph{angular measure} on $S$ is a nonatomic
Borel measure that is invariant under the antipodal map and has total
mass $2\pi$. It is called a \emph{Brass measure} if every angle of a
norm-equilateral triangle has measure $\pi/3$, and it is called
\emph{proper} if every nontrivial arc of $S$ has positive measure.

Every strictly convex normed plane admits a proper Brass
measure. Indeed, this follows by combining
\cite[Theorem~15]{SwanepoelBrass} with
\cite[Lemma~14]{SwanepoelBrass}: strict convexity implies that
the unit circle contains no nontrivial line segment and hence that
the parameter $\lambda(\mathbb V)$ appearing there is equal to zero. For the remainder of the geometric part of the proof of Theorem~\ref{thm:main}, we fix such a measure $\mu$ on
the unit circle $S$.

We first record a normalisation observation.

\begin{lemma}\label{lem:normalization}
Suppose that $v,u\in \mathbb V$ satisfy $0<\|u\|<\|v\|$ and
$\|u-v\|>\|v\|$.
Then, for $U=u/\|u\|$ and $V=v/\|v\|$, one has
$\|U-V\|>1$.
\end{lemma}

\begin{proof}
Set $\lambda=\frac{\|u\|}{\|v\|}\in(0,1)$.
Suppose, for contradiction, that $\|U-V\|\leq1$. Since
$\lambda U-V=\lambda(U-V)+(1-\lambda)(-V)$,
the convexity of the norm gives
\[
\begin{aligned}
    \frac{\|u-v\|}{\|v\|}
      &=\|\lambda U-V\| \\
      &\leq \lambda\|U-V\|+(1-\lambda)\|V\| \\
      &\leq \lambda+(1-\lambda)=1,
\end{aligned}
\]
where we used $\|V\|=1$. Hence $\|u-v\|\leq\|v\|$, contrary to the
hypothesis. 
\end{proof}

We shall also need the following monotonicity property of chords.

\begin{lemma}[Chord monotonicity]\label{lem:monotonicity}
Let $a,b,c\in\mathbb V\setminus\{0\}$, with $a\neq c$, and suppose that
$b$ and $c$ lie on the same circle centred at the origin, that is,
$\|b\|=\|c\|$. If the ray $Ob$
lies in the angle between the rays $Oa$ and $Oc$, whose size is at most
$\pi$, then $\|a-b\|\leq \|a-c\|$.
\end{lemma}

\begin{proof}
Assume first that $Ob$ is distinct from the two boundary rays.
The ray $Ob$ intersects the segment $[a,c]$ at a point $p$.
This remains true when the angle is equal to $\pi$, in which case
$p=O$.

Since $O,p,b$ lie on the same ray, either $p\in[O,b]$ or
$b\in[O,p]$, see Figure~\ref{fig:monotonicity}.

If $p\in[O,b]$, then
\[
\begin{aligned}
 \|b\|+\|a-c\|
 &=\bigl(\|p\|+\|p-b\|\bigr)
   +\bigl(\|a-p\|+\|p-c\|\bigr)\\
 &\geq \|c\|+\|a-b\|.
\end{aligned}
\]
Since $\|b\|=\|c\|$, this gives
$\|a-b\|\leq\|a-c\|$.

If $b\in[O,p]$, then
\[
 \|p-b\|=\|p\|-\|b\|
 \leq\|p-c\|+\|c\|-\|b\|
 =\|p-c\|,
\]
and hence
\[
 \|a-b\|
 \leq\|a-p\|+\|p-b\|
 \leq\|a-p\|+\|p-c\|
 =\|a-c\|.
\]

If $Ob=Oa$, then $a$ and $b$ lie on the same ray. The reverse triangle inequality gives
\[
    \|a-c\|
    \geq \bigl|\|a\|-\|c\|\bigr|
    =\bigl|\|a\|-\|b\|\bigr|
    =\|a-b\|.
\]

If $Ob=Oc$, then $\|b\|=\|c\|$ implies $b=c$. This proves the required
inequality.
\end{proof}

\begin{figure}[htbp]
    \centering
    \begin{tikzpicture}[
    scale=1.7,
    point/.style={
        circle,
        fill=cyan!30!black,
        inner sep=1.1pt
    },
    every node/.append style={
        text=cyan!30!black
    }
]

    \def\radius{1.2}

    \coordinate (O1) at (-2,0);
    \coordinate (O2) at (2,0);

    \coordinate (A1) at (-4.3,0);
    \coordinate (A2) at (-0.3,0);

    \coordinate (C1) at ($(O1)+(55:\radius)$);
    \coordinate (B1) at ($(O1)+(120:\radius)$);

    \coordinate (B2) at ($(O2)+(160:\radius)$);
    \coordinate (C2) at ($(O2)+(100:\radius)$);

    \path[name path=lineA1C1] (A1) -- (C1);
    \path[name path=lineO1B1] (O1) -- (B1);

    \path[
        name intersections={of=lineA1C1 and lineO1B1, by=P1}
    ];

    \path[name path=lineA2C2]
        ($(A2)!-1.5!(C2)$) --
        ($(A2)!1.5!(C2)$);

    \path[name path=lineO2B2]
        ($(O2)!-1.5!(B2)$) --
        ($(O2)!1.5!(B2)$);

    \path[
        name intersections={of=lineA2C2 and lineO2B2, by=P2}
    ];

    \draw[very thick, magenta!50!black]
        (A1) -- (O1);
    \draw[very thick, magenta!50!black]
        (A2) -- (O2);

    \draw[very thick, magenta!50!black]
        (O1) circle (\radius);
    \draw[very thick, magenta!50!black]
        (O2) circle (\radius);

    \draw[very thick, cyan!60!black]
        (A1) -- (B1);
    \draw[very thick, cyan!60!black]
        (A1) -- (C1);
    \draw[very thick, magenta!50!black]
        (O1) -- (B1);
    \draw[very thick, magenta!50!black]
        (O1) -- (C1);

    \draw[very thick, cyan!60!black]
        (A2) -- (B2);
    \draw[very thick, cyan!60!black]
        (A2) -- (C2);
    \draw[very thick, magenta!50!black]
        (O2) -- (P2);
    \draw[very thick, magenta!50!black]
        (O2) -- (C2);

    \node[point, label=below:$O$] at (O1) {};
    \node[point, label=below:$O$] at (O2) {};

    \node[point, label=below:$a$] at (A1) {};
    \node[point, label=below:$a$] at (A2) {};

    \node[point, label=above right:$c$] at (C1) {};
    \node[point, label=above left:$b$] at (B1) {};

    \node[point] at (B2) {};
    \node[right, xshift=3pt, yshift=3pt] at (B2) {$b$};

    \node[point, label=below right:$c$] at (C2) {};

    \node[point, label=above right:$p$] at (P1) {};
    \node[point, label=above:$p$] at (P2) {};

\end{tikzpicture} 
    \caption{Illustration of the proof of the chord monotonicity lemma}
    \label{fig:monotonicity}
\end{figure}

The next lemma relates chord lengths to the Brass angular measure.

\begin{lemma}\label{lem:brass-chord}
Let $a,b\in S$. Then
$\|a-b\|\geq1\Longrightarrow
 \mu(\widehat{aOb})\geq\frac{\pi}{3}$.
If $\|a-b\|>1$, then
$\mu(\widehat{aOb})>\frac{\pi}{3}$.
\end{lemma}

\begin{proof}
If $b=-a$, the chosen arc is a semicircle and has measure $\pi$ since $\mu$
is nonatomic and invariant under the antipodal map. Now suppose
that $b\ne-a$. As $q$ moves along the minor arc from $a$ to $b$, the function
$q\longmapsto\|a-q\|$
is nondecreasing by Lemma~\ref{lem:monotonicity}: for two points $q_1,q_2$
occurring in this order, apply that lemma with $a$ fixed and with $b=q_1$ and
$c=q_2$. This function is continuous, equals zero at $q=a$, and is at least
one at $q=b$. Hence there is a point $q$ on the arc such that $\|a-q\|=1$.

The triangle with vertices $O,a,q$ is norm-equilateral:
$\|a\|=\|q\|=\|a-q\|=1$.
By the definition of a Brass measure,
$\mu(\widehat{aOq})=\pi/3$. The arc $\widehat{aOq}$ is contained in
$\widehat{aOb}$, which proves the non-strict inequality. If $\|a-b\|>1$,
then $q\ne b$; the remaining arc from $q$ to $b$ is nontrivial and therefore
has positive measure. Consequently, $\mu(\widehat{aOb})>\pi/3$.
\end{proof}

For $u,v\in\mathbb V$, let
$\pos(u,v)=\{su+tv:s,t\geq0\}$
denote the closed positive cone generated by $u$ and $v$.

\begin{lemma}[Cone lemma]\label{lem:cone}
Let $x,y,z\in \mathbb V\setminus\{0\}$ satisfy
\begin{align}
  0<\|x\|<\|y\|<\|z\|,                                      \\
  x\in\pos(y,z),                                            \\
  \|x-y\|>\|y\|,\qquad \|x-z\|>\|z\|.                       \label{eq:separation}
\end{align}
Then $\|x-y-z\|<\|z\|$.

\end{lemma}

\begin{proof}
Set
\[
 r=\|x\|,\qquad s=\|y\|,\qquad t=\|z\|,
 \qquad X=\frac{x}{r},\quad Y=\frac{y}{s},\quad Z=\frac{z}{t}.
\]
Applying Lemma~\ref{lem:normalization} to the pairs $(x,y)$ and $(x,z)$ gives
\begin{equation}\label{eq:normalized-separation}
  \|X-Y\|>1,\qquad \|X-Z\|>1.
\end{equation}

We first show that $y$ and $z$ are linearly independent and that $x$
is collinear with neither of them. If $x$ were on the same ray as $y$, then $r<s$ would imply $\|x-y\|=s-r<s$, contrary to \eqref{eq:separation}. The case of $z$ is
analogous. Write $x=\alpha y+\beta z$, where $\alpha,\beta\geq0$. The
preceding observation implies that $\alpha,\beta>0$. Moreover, if $y$
and $z$ were linearly dependent, then $x\in\pos(y,z)$ would lie on the
same ray as at least one of them, again a contradiction. Thus $y$ and
$z$ are linearly independent, and the ray $OX$ lies strictly between
$OY$ and $OZ$ inside $\pos(y,z)$. The corresponding arcs
$\widehat{YOX}$ and $\widehat{XOZ}$ are therefore minor arcs.

Applying Lemma~\ref{lem:brass-chord} to these two minor arcs,
\eqref{eq:normalized-separation} gives
\[
  \mu(\widehat{YOX})>\frac{\pi}{3},\qquad
  \mu(\widehat{XOZ})>\frac{\pi}{3}.
\]
Let $\theta=\mu(\widehat{YOX})+\mu(\widehat{XOZ})$
be the measure of the arc from $Y$ to $Z$ that contains $X$. Since $x$
belongs to the positive cone, this arc is contained in a semicircle; hence
\[
  \frac{2\pi}{3}<\theta<\pi.
\]
The measures may be added in the definition of $\theta$ because the two arcs
intersect only at $X$ and the angular measure is nonatomic. The strict 
inequality $\theta<\pi$ follows because the complementary arc in this
semicircle is nontrivial and therefore has positive measure.

The complementary arc from $Z$ to $-Y$ in the same semicircle has measure
\[
 \mu(\widehat{ZO(-Y)})=\pi-\mu(\widehat{YOX})
 -\mu(\widehat{XOZ})=\pi-\theta<\frac{\pi}{3}.
\]
Here we specifically choose the complementary arc within the semicircle from
$Y$ to $-Y$. The three arcs partition this semicircle, and their common
endpoints have measure zero. By the contrapositive of the non-strict part of
Lemma~\ref{lem:brass-chord},

\begin{equation}\label{eq:YZ-chord}
  \|Z-(-Y)\|=\|Y+Z\|<1.
\end{equation}

Set $A=y+z$. From \eqref{eq:YZ-chord},
\[
 A=sY+tZ=(t-s)Z+s(Y+Z),
\]
and consequently
\begin{equation}\label{eq:A-bound}
  \|A\|\leq(t-s)\|Z\|+s\|Y+Z\|<t.
\end{equation}

The ray $OA$ lies inside the same cone $\pos(y,z)$. The point $X$ divides the
angle between $Y$ and $Z$ into two parts. Thus either $OX$ lies between $OA$
and $OY$, or it lies between $OA$ and $OZ$ (the case $OX=OA$ is handled
directly). Applying Lemma~\ref{lem:monotonicity} on the circle of radius $r$
gives
\begin{equation}\label{eq:monotonicity-bound}
  \|A-x\|\leq
  \max\{\|A-rY\|,\|A-rZ\|\}.
\end{equation}
Here $a=A$, $b=x$, and $c=rY$ or $c=rZ$. The linear independence of $y,z$
proved above implies $A\ne0$ and $A\ne rY,rZ$. Moreover,
$\|b\|=\|c\|=r$, and the relevant ray $Ob$ lies between $Oa$ and $Oc$ in an
angle smaller than $\pi$. Thus all the hypotheses of Lemma~\ref{lem:monotonicity}  are
satisfied.

If $OX=OA$, then \eqref{eq:A-bound} and $r<t$ immediately give
$\|A-x\|=|\|A\|-r|<t$.

It remains to estimate the right-hand side of
\eqref{eq:monotonicity-bound}. We have
\begin{align*}
 A-rY&=\frac r s z+\left(1-\frac r s\right)A,\\
 A-rZ&=\frac r t y+\left(1-\frac r t\right)A.
\end{align*}
Using $s<t$, \eqref{eq:A-bound}, and $0<r<s<t$, we obtain
\begin{align*}
 \|A-rY\|
 &\leq \frac r s\,t+\left(1-\frac r s\right)\|A\|<t,\\
 \|A-rZ\|
 &\leq \frac r t\,s+\left(1-\frac r t\right)\|A\|<t.
\end{align*}
Together with \eqref{eq:monotonicity-bound}, this gives
\[
  \|x-y-z\|=\|A-x\|<t=\|z\|.
\]
\end{proof}

\begin{remark}
Figure~\ref{fig:drawing} provides a geometric interpretation of the
estimate following the application of Lemma~\ref{lem:monotonicity}.

That lemma gives
$\|A-x\|\leq\left\|A-\frac{r}{s}y\right\|$.

\begin{figure}[htbp]
    \centering
    \begin{tikzpicture}[
    scale=2.1,
    point/.style={fill=cyan!20!black},
    every node/.append style={text=cyan!20!black}
]

    \coordinate (O) at (0,0);
    \def\rsmall{1.2}
    \def\rbig{2.4}
    \def\angle{138}

    \coordinate (C) at (\rsmall,0);
    \coordinate (y) at (2.2,0);

    \coordinate (B) at (\angle:\rsmall);
    \coordinate (z) at (\angle:\rbig);

    \coordinate (x) at (62:\rsmall);
    \coordinate (A) at ($(y)+(z)-(O)$);

    \path
        let
            \p1 = ($(C)-(O)$),
            \p2 = ($(y)-(C)$)
        in
            \pgfextra{
                \pgfmathsetmacro{\tempRatio}{
                    veclen(\x1,\y1) /
                    (veclen(\x1,\y1) + veclen(\x2,\y2))
                }
                \global\let\divisionRatio\tempRatio
            };

    \coordinate (D) at ($(O)!\divisionRatio!(A)$);

    \draw[cyan!50!black, very thick]
        (O) circle (\rsmall);

    \draw[magenta!40!black, very thick, dashed]
        (O) circle (\rbig);

    \draw[cyan!40!black, thick]
        (O) -- (3.2,0);

    \draw[cyan!40!black, thick]
        (O) -- (\angle:3.8);

    \draw[dashed, magenta!60!black]
        (A) -- (z);

    \draw[dashed, magenta!60!black]
        (A) -- (y)
        node[midway, above, magenta!30!black] {$\boldsymbol{z}$};

    \draw[dashed, thick, cyan!40!black]
        (O) -- (A);

    \draw[cyan!40!black, thick]
        (O) -- (x);

    \draw[magenta!60!black, very thick]
        (A) -- (x);

    \draw[magenta!60!black, very thick]
        (C) -- (D)
        node[midway, below left, inner sep=1pt, magenta!30!black]
        {$\boldsymbol{\frac{r}{s}z}$};

    \draw[
        magenta!60!black,
        very thick,
        line join=round
    ]
        (C) -- (D) -- (A) -- cycle;

    \fill[
        magenta!60!black,
        fill opacity=0.2,
        rounded corners=2pt
    ]
        (C) -- (D) -- (A) -- cycle;

    \fill[point] (C) circle (0.7pt)
        node[below right] {$\boldsymbol{\frac{r}{s}y}$};

    \fill[point] (y) circle (0.7pt)
        node[below] {$\boldsymbol{y}$};

    \fill[point] (x) circle (0.7pt)
        node[below, cyan!15!black, inner sep=6pt] {$\boldsymbol{x}$};

    \fill[point] (B) circle (0.7pt)
        node[left] {$\boldsymbol{\frac{r}{t}z}$};

    \fill[point] (z) circle (0.7pt)
        node[below, inner sep=6pt] {$\boldsymbol{z}$};

    \fill[point] (O) circle (0.7pt)
        node[below right] {$\boldsymbol{O}$};

    \fill[point] (A) circle (0.7pt)
        node[right, inner sep=6pt] {$\boldsymbol{A}$};

    \fill[point] (D) circle (0.7pt)
        node[left] {$\boldsymbol{\frac{r}{s}A}$};

\end{tikzpicture}
    \caption{Illustration of the proof of the Cone lemma in the case when the ray $OX$ lies between $OA$ and $OY$}
    \label{fig:drawing}
\end{figure}

The remaining estimate is simply the triangle inequality applied to
the shaded triangle in Figure~\ref{fig:drawing}. Its vertices
are $A$, $\frac{r}{s}y$, and $\frac{r}{s}A$.
The side opposite $A$ is parallel to $z$, since
\[
    \frac{r}{s}A-\frac{r}{s}y
    =\frac{r}{s}(A-y)
    =\frac{r}{s}z,
\]
and therefore its length is $(r/s)t$. The other side lies on the ray $OA$ and has length
$\left\|A-\frac{r}{s}A\right\|
    =\left(1-\frac{r}{s}\right)\|A\|$.
Consequently,
\[
\begin{aligned}
    \|A-x\|
    &\leq
    \left\|A-\frac{r}{s}y\right\|\\
    &\leq
    \left\|A-\frac{r}{s}A\right\|
    +
    \left\|\frac{r}{s}A-\frac{r}{s}y\right\|\\
    &=
    \left(1-\frac{r}{s}\right)\|A\|
    +\frac{r}{s}t
    <t,
\end{aligned}
\]
where the final inequality follows from $\|A\|<t$ and $0<r/s<1$.

The case in which $OX$ lies between $OA$ and $OZ$ is analogous:
one uses the triangle with vertices
$A$, $\frac{r}{t}z$, and $\frac{r}{t}A$,
whose side opposite $A$ is parallel to $y$.
\end{remark}

\section{Proof of Theorem~\ref{thm:main}}

We divide the proof into two parts. First, assuming that the norm is
strictly convex, we associate a lattice with the finite Kronecker
sequence and bound the number of values attained by the corresponding
lattice minimum function. This part follows the lattice-theoretic
argument of Haynes and Marklof, with the necessary geometric input
supplied by the preceding section. We then remove the strict convexity
assumption by perturbing an arbitrary norm with a Euclidean component
and passing to the limit.

\subsection{The strictly convex case}

Throughout this subsection, the norm $\|\cdot\|$ is assumed to be
strictly convex. Since the norm is fixed, we suppress it from the
notation and write
$\delta_{n,N}:=\delta_{n,N}^{\|\cdot\|}$ and
$g_N:=g_N^{\|\cdot\|}(\alpha,L)$.

Setting $k=m-n$ in \eqref{eq:delta-definition} and using the central
symmetry of the norm, we obtain
\begin{equation}\label{eq:delta-original}
\delta_{n,N}
=
\min\bigl\{
  \|k\alpha+\ell\|:
  -n<k\leq N-n,\ \ell\in L,\
  k\alpha+\ell\ne0
\bigr\}.
\end{equation}

Write $\widehat N=N+\frac12$ and set
\[
\Lambda_N=
\left\{
\left(
  \frac{k}{\widehat N},
  k\alpha+\ell
\right):
k\in\mathbb Z,\ \ell\in L
\right\}
\subset\R\times\R^2.
\]
Since both $\widehat N^{-1}\mathbb Z\subset\R$ and $L\subset\R^2$ are
full-rank lattices, so is their product. The set $\Lambda_N$ is
obtained from this product by a linear shear and is therefore
also full-rank.

For $0<t<1$, define
\begin{align*}
 Q_{\Lambda_N}(t)
   &=\{(u,v)\in\Lambda_N:-t<u<1-t,\ v\ne0\},\\
 F_{\Lambda_N}(t)
   &=\min\{\|v\|:(u,v)\in Q_{\Lambda_N}(t)\}.
\end{align*}
Since $\Lambda_N$ is fixed throughout the remainder of this subsection,
we write $Q(t)=Q_{\Lambda_N}(t)$ and $F(t)=F_{\Lambda_N}(t)$.

Choose a nonzero $\ell_0\in L$ and put $C=\|\ell_0\|$. Since
$(0,\ell_0)\in Q(t)$ for every $t\in(0,1)$, it is enough to
minimise over the points of $Q(t)$ whose spatial components
have norm at most $C$. These points form a nonempty finite set, since
they belong to the compact set
$\{(u,v):|u|\leq1,\ \|v\|\leq C\}$. Thus $F(t)$ is well
defined and positive for every $t\in(0,1)$. Moreover, all the values of $F$ are obtained from this fixed finite set,
so $\{F(t):0<t<1\}$ is finite. Set $K:=\#\{F(t):0<t<1\}$.

\begin{lemma}\label{lem:exact-reduction}
For every $1\leq n\leq N$,
\begin{equation}\label{eq:exact-delta-F}
\delta_{n,N}
=
F\left(\frac{n}{\widehat N}\right).
\end{equation}
Consequently,
\begin{equation}\label{eq:g-le-K}
g_N
=
\#\left\{
F\left(\frac{n}{\widehat N}\right):
1\leq n\leq N
\right\}
\leq K.
\end{equation}
\end{lemma}

\begin{proof}
Let $t_n=n/\widehat N$. A point
$(k/\widehat N,k\alpha+\ell)\in\Lambda_N$ belongs to
$Q(t_n)$ precisely when
$-n<k<N-n+\frac12$ and
$k\alpha+\ell\ne0$. Since $k$ is an integer, the first
condition is equivalent to $-n<k\leq N-n$. Thus the spatial components
of the points in $Q(t_n)$ are exactly the admissible vectors
in \eqref{eq:delta-original}, which proves \eqref{eq:exact-delta-F}.
Taking the number of distinct values gives \eqref{eq:g-le-K}.
\end{proof}

Choose $(u_i,v_i)\in\Lambda_N$, $1\leq i\leq K$, so that
\begin{enumerate}
 \item $0<\|v_1\|<\cdots<\|v_K\|$;
 \item the set $\{\|v_i\|:1\leq i\leq K\}$ is precisely the range of $F$;
 \item for every $i$ there exists $t_i\in(0,1)$ such that
 $(u_i,v_i)\in Q(t_i)$ and $F(t_i)=\|v_i\|$.
\end{enumerate}
Since $Q(1-t)=-Q(t)$, we have $F(1-t)=F(t)$. Hence, replacing
$(u_i,v_i)$ by $(-u_i,-v_i)$ and $t_i$ by $1-t_i$ when necessary, we may
assume that $u_i\geq0$.

\begin{lemma}\label{lem:lattice-monotonicity}
If $(u,v)\in\Lambda_N$, $v\ne0$, and
$0\leq u\leq u_i<1$,
then $\|v_i\|\leq\|v\|$.
\end{lemma}

\begin{proof}
Choose $t_i$ as in the third condition. Since $(u_i,v_i)\in Q(t_i)$ and
$u_i\geq0$, we have $t_i<1-u_i\leq1-u$, so $(u,v)\in Q(t_i)$. Therefore
$\|v_i\|=F(t_i)\leq\|v\|$.
\end{proof}

\begin{lemma}\label{lem:small-u}
If $(u_i,v_i)$ is one of the chosen representatives and
$0\leq u_i<1/2$, then $i=K$.
\end{lemma}

\begin{proof}
For $0<t<1-u_i$, the vector $(u_i,v_i)$ belongs to $Q(t)$, whereas for
$u_i<t<1$, the vector $(-u_i,-v_i)$ belongs to $Q(t)$. Since $u_i<1/2$,
these two intervals cover $(0,1)$. Hence $F(t)\leq\|v_i\|$ for every
$t\in(0,1)$. Thus $\|v_i\|$ is the largest value attained by $F$, and so
$i=K$.
\end{proof}

\begin{proposition}\label{prop:lattice-inequalities}
For all $1\leq i<j\leq K-1$, one has
\begin{equation}\label{eq:representative-pairwise}
 \|v_i-v_j\|>\|v_j\|,
\end{equation}
and for all $1\leq i<j<k\leq K-1$, one has
\begin{equation}\label{eq:representative-triple}
 \|v_i-v_j-v_k\|\geq\|v_k\|.
\end{equation}
\end{proposition}

\begin{proof}
Lemma~\ref{lem:small-u} gives $u_i\geq1/2$ for $i<K$. We also claim that
$u_i>u_j$ whenever $i<j<K$. If $u_i\leq u_j$,
Lemma~\ref{lem:lattice-monotonicity} gives
$\|v_i\|\geq\|v_j\|$, contrary to
$\|v_i\|<\|v_j\|$. Thus
$0<u_i-u_j<\frac12\leq u_j$.
Applying Lemma~\ref{lem:lattice-monotonicity} to the lattice vector
$(u_i-u_j,v_i-v_j)$ and the representative $(u_j,v_j)$ gives
\begin{equation}\label{eq:pairwise-nonstrict}
 \|v_i-v_j\|\geq\|v_j\|.
\end{equation}
If equality holds in \eqref{eq:pairwise-nonstrict}, then the vector
$(u_i-u_j,v_i-v_j)$ and its
negative show, exactly as in the proof of Lemma~\ref{lem:small-u}, that
$F(t)\leq\|v_j\|$ for all $t\in(0,1)$. This would imply $j=K$, which is
impossible. This proves \eqref{eq:representative-pairwise}.

Now let $i<j<k\leq K-1$. The ordering of the $u$-coordinates established
above implies
\begin{equation}\label{eq:time-triple}
 \frac12\leq u_k<u_j<u_i<1,
 \qquad 0<u_k+u_j-u_i<u_k.
\end{equation}
The spatial component $v_k+v_j-v_i$ is nonzero. Indeed, otherwise
$v_i=v_j+v_k$, and \eqref{eq:representative-pairwise}, applied to the
pair $(i,k)$, would give
$\|v_j\|=\|v_i-v_k\|>\|v_k\|$,
contrary to $j<k$. Hence
$(u_k+u_j-u_i,v_k+v_j-v_i)\in\Lambda_N$ is an admissible lattice vector.
From \eqref{eq:time-triple} and
Lemma~\ref{lem:lattice-monotonicity}, applied
to the representative $(u_k,v_k)$, we obtain
$\|v_k+v_j-v_i\|\geq\|v_k\|$.
The central symmetry of the norm yields
\eqref{eq:representative-triple}.
\end{proof}

\begin{proposition}\label{prop:exclusion}
Let nonzero vectors $v_1,\dots,v_m$ satisfy
$\|v_1\|<\|v_2\|<\cdots<\|v_m\|$
and suppose that, for all applicable indices,
\begin{align}
  \|v_i-v_j\|&>\|v_j\| &&(i<j),                         \label{eq:pairwise}\\
  \|v_i-v_j-v_k\|&\geq\|v_k\| &&(i<j<k).               \label{eq:triple}
\end{align}
Then, whenever $i<j<k$, the vector $v_i$ does not belong to the cone
$\pos(v_j,v_k)$.
\end{proposition}

\begin{proof}
If $v_i\in\pos(v_j,v_k)$, then Lemma~\ref{lem:cone}, applied with
$x=v_i$, $y=v_j$, and $z=v_k$, would give
$\|v_i-v_j-v_k\|<\|v_k\|$,
contrary to \eqref{eq:triple}.
\end{proof}

\begin{lemma}\label{lem:five-vectors}
Under the hypotheses of Proposition~\ref{prop:exclusion}, the number of
vectors is at most four: $m\leq4$.
\end{lemma}

\begin{proof}
Suppose, to the contrary, that $m\geq5$. Consider the first five vectors
and set $a_i=\frac{v_i}{\|v_i\|}$, $1\leq i\leq5$.
For $i<j$, \eqref{eq:pairwise} and Lemma~\ref{lem:normalization} imply
\begin{equation}\label{eq:normalized-pairwise}
  \|a_i-a_j\|>1.
\end{equation}

For each pair of cyclically consecutive points, the corresponding
elementary arc is a minor arc, a semicircle, or a major arc. In the first
case, \eqref{eq:normalized-pairwise} and
Lemma~\ref{lem:brass-chord} imply that its measure is strictly greater
than $\pi/3$; in the other two cases, its measure is at least $\pi$.
Hence every elementary arc between consecutive points has measure
strictly greater than $\pi/3$.

Let $a_j$ and $a_k$ be the two cyclic neighbours of $a_1$. The three
elementary arcs forming the arc from $a_j$ to $a_k$ that does not contain
$a_1$ have total measure strictly greater than $\pi$. Their complementary
arc, which contains $a_1$, therefore has measure strictly less than $\pi$
and is the minor arc from $a_j$ to $a_k$. Hence the ray $Ov_1$ lies in
the positive cone bounded by the rays $Ov_j$ and $Ov_k$:
$v_1\in\pos(v_j,v_k)$.
After interchanging $j$ and $k$ if necessary, we may assume that $1<j<k$.
This contradicts Proposition~\ref{prop:exclusion}. Therefore $m\leq4$.
\end{proof}

Proposition~\ref{prop:lattice-inequalities} and
Lemma~\ref{lem:five-vectors} give $K-1\leq4$, and hence $K\leq5$.
Together with Lemma~\ref{lem:exact-reduction}, this yields
$g_N\leq K\leq5$. This proves
Theorem~\ref{thm:main} under the additional assumption that
$\|\cdot\|$ is strictly convex.

\subsection{Removing the strict convexity assumption}

It remains to remove the strict convexity assumption. To this end, we
perturb an arbitrary norm by adding a small Euclidean term. The resulting
norm is strictly convex, and as the perturbation tends to zero, the
corresponding nearest-neighbour distances converge to those for the
original norm; consequently, six distinct distances for the original norm would remain distinct under every sufficiently small perturbation.

Let $\|\cdot\|$ be an arbitrary norm on $\mathbb R^2$. For each $\varepsilon>0$, define a norm by $\|x\|_\varepsilon
\coloneqq \|x\|+\varepsilon |x|_2$, where $|\cdot|_2$ denotes the Euclidean norm.

\begin{lemma}
$\|x\|_\varepsilon$ is strictly convex.
\end{lemma}

\begin{proof}
Suppose, to the contrary, that $\|\cdot\|_\varepsilon$ is not strictly
convex. Then there exist distinct vectors $x$ and $y$ such that
$\|x\|_\varepsilon=\|y\|_\varepsilon=1$ and
$\left\|\frac{x+y}{2}\right\|_\varepsilon=1$.
By positive homogeneity,
$\|x+y\|_\varepsilon
=2=\|x\|_\varepsilon+\|y\|_\varepsilon$,
so equality holds in the triangle inequality for
$\|\cdot\|_\varepsilon$. Since this inequality is obtained by adding
the triangle inequalities for $\|\cdot\|$ and
$\varepsilon|\cdot|_2$, equality implies
$|x+y|_2=|x|_2+|y|_2$.
Equality for the Euclidean norm implies that $x$ and $y$
lie on the same ray from the origin. Hence $y=\lambda x$ for some
$\lambda>0$. By positive homogeneity,
$1=\|y\|_\varepsilon=\|\lambda x\|_\varepsilon
      =\lambda\|x\|_\varepsilon=\lambda$.
Therefore $\lambda=1$ and $x=y$, contradicting their choice as distinct
vectors. Hence $\|\cdot\|_\varepsilon$ is strictly convex.
\end{proof}

For each $n$, let $\mathcal D_{n,N}$ denote the norm-independent set of nonzero displacement vectors
over which the minimum in \eqref{eq:delta-original} is taken. We write
$\delta_{n,N}^{(\varepsilon)}=\min_{x\in \mathcal D_{n,N}}\|x\|_\varepsilon$ and
$\delta_{n,N}^{(0)}=\min_{x\in \mathcal D_{n,N}}\|x\|$. Choose $x_n\in \mathcal D_{n,N}$ such that
$\|x_n\|=\delta_{n,N}^{(0)}$; its existence follows from the same
local-finiteness argument used above for $F_{\Lambda_N}(t)$. Since $\|x\|\leq\|x\|_\varepsilon$ for every $x$, we have
\[
    \delta_{n,N}^{(0)}
    \leq\delta_{n,N}^{(\varepsilon)}
    \leq\|x_n\|_\varepsilon
    =\delta_{n,N}^{(0)}+\varepsilon|x_n|_2.
\]
Hence $\delta_{n,N}^{(\varepsilon)}\to\delta_{n,N}^{(0)}$ as
$\varepsilon\to0^+$.

Suppose, for contradiction, that the original norm gives six distinct
distances
$\delta_{n_1,N}^{(0)},\ldots,\delta_{n_6,N}^{(0)}$. Let
$\rho=\frac13\min_{r\neq s}
    \left|\delta_{n_r,N}^{(0)}-\delta_{n_s,N}^{(0)}\right|>0$.
For all sufficiently small $\varepsilon>0$, each
$\delta_{n_r,N}^{(\varepsilon)}$ differs from
$\delta_{n_r,N}^{(0)}$ by less than $\rho$. It follows that the six
perturbed distances remain pairwise distinct. This contradicts the
strictly convex case, since $\|\cdot\|_\varepsilon$ is strictly convex
and therefore
$g_N^{\|\cdot\|_\varepsilon}(\alpha,L)\leq5$.
Consequently,
$g_N^{\|\cdot\|}(\alpha,L)\leq5$
for every norm $\|\cdot\|$ on $\mathbb R^2$. This completes the proof of Theorem~\ref{thm:main}.

\section{Proof of Theorem~\ref{thm:sharpness}}

Adapting Dettmann's construction~\cite{Dettmann}, for every norm
$\|\cdot\|$ on $\mathbb R^2$ we will choose linearly independent
unit vectors $a,b$ and a suitable perturbation direction
$h\in\mathbb R^2$.
We will show that the lattice
$L=\mathbb Z(4a-3b)+\mathbb Z(-3a-4b)$ and the vector
$\alpha_\varepsilon=-2b+\varepsilon h$ satisfy
$g_{17}^{\|\cdot\|}(\alpha_\varepsilon,L)=5$
for all sufficiently small $\varepsilon>0$.

As in Section~\ref{sec:brass-measure}, let
\(B=\{x\in\mathbb R^2:\|x\|\leq1\}\) be the unit ball and
\(S=\partial B\) the unit circle, with centre \(O\). We shall use the following facts on
smooth points and Brass angular measures. A point
$x\in S$ is called \emph{smooth} if it has a unique norming functional
$\varphi_x\in(\mathbb R^2)^*$, that is,
$\|\varphi_x\|_*=1$ and $\varphi_x(x)=1$ where $\|\cdot\|_*$ denotes the dual norm. At such a point the norm is
differentiable and, for every $h\in\mathbb R^2$,
$\|x+th\|=1+t\varphi_x(h)+o(t)$ as $t\to0$. Smooth points are dense in
$S$ by Mazur's density theorem; see, for example,
\cite[Section~20]{Holmes}. Unlike in the strictly convex case, an arbitrary normed plane need not
admit a proper Brass measure. Nevertheless, whenever the unit ball $B$
is not a parallelogram, the plane admits a good Brass measure $\mu$ in
the sense of Swanepoel; see
\cite[Theorem~15]{SwanepoelBrass}. Recall that, as an angular measure, $\mu$ is nonatomic and antipodally
invariant, and every semicircle has measure $\pi$. As a Brass measure,
it assigns measure $\pi/3$ to every norm-equilateral angle. We shall
also use that every long segment has measure $\pi/3$ and that an arc
in a semicircle of measure greater than $\pi/3$ has endpoints at norm
distance greater than one; see
\cite[Lemmas~13 and~18]{SwanepoelBrass}. Here a \emph{long segment} is an inclusion-maximal segment
$[u,v]\subset S$ with $\|u-v\|>1$.

\begin{lemma}\label{lem:sharpness-directions}
There exist linearly independent smooth points $a,b\in S$ whose norming
functionals are linearly independent and such that
\begin{equation}\label{eq:sharpness-lattice-clean}
    B\cap(\mathbb Za+\mathbb Zb)=\{0,\pm a,\pm b\}.
\end{equation}
\end{lemma}

\begin{proof}
Suppose first that $B$ is not a parallelogram; see the left panel of
Figure~\ref{fig:sharpness}. And let $\mu$ be the
good Brass measure fixed above. Choose a smooth point $a\in S$ and one
of the semicircles from $a$ to $-a$. As $b$ traverses this semicircle,
the quantity $\mu(\widehat{aOb})$ varies continuously from $0$ to
$\pi$, since $\mu$ is nonatomic. By the density of smooth points and by antipodal invariance, we
may therefore choose a smooth point $b$ such that
\begin{equation}\label{eq:sharpness-angle-bounds}
\frac{\pi}{3}<\mu(\widehat{aOb})<\frac{2\pi}{3},     
\qquad
\frac{\pi}{3}<\mu(\widehat{aO(-b)}) =\pi-\mu(\widehat{aOb})<\frac{2\pi}{3}.
\end{equation}

The properties of $\mu$ stated above now give
$\|a-b\|>1$ and $\|a+b\|>1$.

We next show that $\varphi_a$ and $\varphi_b$ are linearly independent. Since both have dual norm one, linear dependence would
give either $\varphi_b=\varphi_a$ or $\varphi_b=-\varphi_a$. In the first
case, $a$ and $b$ lie in the same exposed face of $B$. Since
$\|a-b\|>1$, this face is a long segment and hence has
$\mu$-measure $\pi/3$. The minor arc $\widehat{aOb}$ is the segment
$[a,b]$ contained in this face, and therefore
$\mu(\widehat{aOb})\leq\pi/3$, contradicting
\eqref{eq:sharpness-angle-bounds}. The second
case is identical with $b$ replaced by $-b$.

Suppose now that $B$ is a parallelogram. After applying a linear
isometry, we may work in $\ell_\infty^2$, see the right panel of
Figure~\ref{fig:sharpness}. Fix $0<t<1$ and set
\[
    a=(1,t),\qquad b=(-t,1).
\]
These points lie in the relative interiors of two adjacent sides of
the unit square and are therefore smooth, with the two coordinate
functionals as their norming functionals. Thus $a$ and $b$ are linearly independent, as are their norming
functionals. Moreover,
\[
    \|a-b\|_\infty=\|a+b\|_\infty=1+t>1.
\]

In both cases we have obtained linearly independent smooth points $a,b$ with
linearly independent norming functionals and
\begin{equation}\label{eq:sharpness-two-chords}
    \|a+b\|>1,\qquad \|a-b\|>1.
\end{equation}
To prove \eqref{eq:sharpness-lattice-clean}, consider
$ma+nb$ with $m,n\in\mathbb Z$. If both coefficients are
nonzero, changing signs and interchanging $a$ and $b$
allows us to assume that $m\ge n\ge1$, while preserving
\eqref{eq:sharpness-two-chords}. The reverse triangle
inequality gives
\[
    \|ma+nb\|
    \ge m\|a+b\|-(m-n)\|b\|
    >n\ge1.
\]
If either coefficient is zero, homogeneity leaves only
$0,\pm a,\pm b$ in $B$. 
\end{proof}

\begin{figure}[htbp]
    \centering
    \begin{tikzpicture}[
    scale=1.1,
    axis/.style={cyan!45!black, very thick},
    figure/.style={magenta!55!black, very thick},
    point/.style={fill=cyan!30!black},
    every node/.append style={text=cyan!30!black}
]

    \begin{scope}[xshift=-3.5cm]

\begin{scope}
    \clip (-3,-2.3) rectangle (3,2.3);

    \begin{scope}[rotate=28]
        \foreach \n in {-3,...,3} {
            \draw[cyan!35!black, thin, opacity=.55]
                ({-3*2.0*cos(20)+\n*2.0*cos(120)},
                 {-3*1.2*sin(20)+\n*1.2*sin(120)})
                --
                ({3*2.0*cos(20)+\n*2.0*cos(120)},
                 {3*1.2*sin(20)+\n*1.2*sin(120)});
        }

        \foreach \m in {-3,...,3} {
            \draw[cyan!35!black, thin, opacity=.55]
                ({\m*2.0*cos(20)-3*2.0*cos(120)},
                 {\m*1.2*sin(20)-3*1.2*sin(120)})
                --
                ({\m*2.0*cos(20)+3*2.0*cos(120)},
                 {\m*1.2*sin(20)+3*1.2*sin(120)});
        }

        \foreach \m in {-3,...,3} {
            \foreach \n in {-3,...,3} {
                \fill[cyan!45!black]
                    ({\m*2.0*cos(20)+\n*2.0*cos(120)},
                     {\m*1.2*sin(20)+\n*1.2*sin(120)})
                    circle (0.55pt);
            }
        }
    \end{scope}
\end{scope}

        \draw[axis] (-3,0) -- (3,0);
        \draw[axis] (0,-2.3) -- (0,2.3);

        \draw[figure, rotate=28]
    (0,0) ellipse [x radius=2.0, y radius=1.2];

\begin{scope}
    \clip (-3,-2.3) rectangle (3,2.3);

    \begin{scope}[rotate=28]

        \draw[dashed, thick, magenta!55!black]
            ({2.0*cos(20)-8*2.0*sin(20)},
             {1.2*sin(20)+8*1.2*cos(20)})
            --
            ({2.0*cos(20)+8*2.0*sin(20)},
             {1.2*sin(20)-8*1.2*cos(20)});

        \draw[dashed, thick, magenta!55!black]
            ({2.0*cos(120)-8*2.0*sin(120)},
             {1.2*sin(120)+8*1.2*cos(120)})
            --
            ({2.0*cos(120)+8*2.0*sin(120)},
             {1.2*sin(120)-8*1.2*cos(120)});

    \end{scope}
\end{scope}

\node[below, text=magenta!55!black] at (2.3,1.3)
    {$\varphi_a$};

\node[left, text=magenta!55!black] at (-0.8,1.2)
    {$\varphi_b$};

\begin{scope}[rotate=28]
    \coordinate (A1)  at ({2.0*cos(20)},  {1.2*sin(20)});
    \coordinate (B1)  at ({2.0*cos(120)}, {1.2*sin(120)});
    \coordinate (-A1) at ({2.0*cos(200)}, {1.2*sin(200)});
    \coordinate (-B1) at ({2.0*cos(300)}, {1.2*sin(300)});
\end{scope}

\fill[point] (0,0) circle (1.4pt)
    node[below right] {$O$};
\fill[point] (A1)  circle (1.4pt) node[above right] {$a$};
\fill[point] (B1)  circle (1.4pt) node[above left]  {$b$};
\fill[point] (-A1) circle (1.4pt) node[below left]  {$-a$};
\fill[point] (-B1) circle (1.4pt) node[below right] {$-b$};
\end{scope}

    \begin{scope}[xshift=3.5cm]

\begin{scope}
    \clip (-3,-2.3) rectangle (3,2.3);

    \foreach \n in {-3,...,3} {
        \draw[cyan!35!black, thin, opacity=.55]
            ({-3*1.44+\n*(-0.36)},
             {-3*0.36+\n*1.44})
            --
            ({3*1.44+\n*(-0.36)},
             {3*0.36+\n*1.44});
    }

    \foreach \m in {-3,...,3} {
        \draw[cyan!35!black, thin, opacity=.55]
            ({\m*1.44-3*(-0.36)},
             {\m*0.36-3*1.44})
            --
            ({\m*1.44+3*(-0.36)},
             {\m*0.36+3*1.44});
    }

    \foreach \m in {-3,...,3} {
        \foreach \n in {-3,...,3} {
            \fill[cyan!45!black]
                ({\m*1.44+\n*(-0.36)},
                 {\m*0.36+\n*1.44})
                circle (0.55pt);
        }
    }
\end{scope}

        \draw[axis] (-3,0) -- (3,0);
        \draw[axis] (0,-2.3) -- (0,2.3);

        \draw[figure]
    (-1.44,-1.44) rectangle (1.44,1.44);

\coordinate (A)  at (1.44,0.36);
\coordinate (B)  at (-0.36,1.44);
\coordinate (-A) at (-1.44,-0.36);
\coordinate (-B) at (0.36,-1.44);

\draw[dashed, thick, magenta!55!black]
    (1.44,1.44) -- (1.44,2.3);

\draw[dashed, thick, magenta!55!black]
    (1.44,-1.44) -- (1.44,-2.3);

\draw[dashed, thick, magenta!55!black]
    (-1.44,1.44) -- (-3,1.44);

\draw[dashed, thick, magenta!55!black]
    (1.44,1.44) -- (3,1.44);

\node[below, text=magenta!55!black] at (1.7,-1.7)
    {$\varphi_a$};

\node[left, text=magenta!55!black] at (-1.9,1.65)
    {$\varphi_b$};

    \fill[point] (0,0) circle (1.4pt)
    node[below right] {$O$};
    \fill[point] (A)  circle (1.4pt) node[right] {$a$};
    \fill[point] (B) circle 
    (1.4pt) node[above] {$b$};
    \fill[point] (-A) circle (1.4pt) node[left]  {$-a$};
    \fill[point] (-B) circle (1.4pt) node[below] {$-b$};
    \end{scope}

\end{tikzpicture}
    \caption{Illustration of Lemma~\ref{lem:sharpness-directions}}
    \label{fig:sharpness}
\end{figure}

We now adapt the rational configuration used by Dettmann
\cite[Section~8.2]{Dettmann}. Put
\[
    u=\frac1{25}(4,-3),\qquad
    v=\frac1{25}(-3,-4),\qquad
    \beta=\frac1{25}(6,8)=-2v .
\]
Let $T:\mathbb R^2\to\mathbb R^2$ be the invertible linear map determined
by $Tu=a$ and $Tv=b$, and set $L=T\mathbb Z^2$ and
$\alpha_0=T\beta$. Since
$e_1=4u-3v$ and $e_2=-3u-4v$, this gives
\begin{equation}\label{eq:sharpness-lattice-alpha0}
    L=\mathbb Z(4a-3b)+\mathbb Z(-3a-4b),
    \qquad
    \alpha_0=-2b .
\end{equation}
The class of $\beta$ has order $25$ in $\mathbb R^2/\mathbb Z^2$, so none
of the cosets $k\beta+\mathbb Z^2$, $1\leq k\leq16$, contains zero. Direct
calculation gives
\[
    9\beta\equiv u,\qquad
    12\beta\equiv v,\qquad
    13\beta\equiv -v,\qquad
    16\beta\equiv -u
    \pmod{\mathbb Z^2}.
\]
These multiples of $\beta$ are distinct modulo $\mathbb Z^2$, so
these are the only cosets containing one of $\pm u,\pm v$.
Since $\beta=-2v$ and $\mathbb Z^2\subset\mathbb Zu+\mathbb Zv$, every
vector in every coset $k\beta+\mathbb Z^2$ belongs to
$\mathbb Zu+\mathbb Zv$. Applying $T$ and using
\eqref{eq:sharpness-lattice-clean}, we obtain, for $1\leq k\leq16$,
\begin{equation}\label{eq:sharpness-initial-minima}
    \min_{\ell\in L}\|k\alpha_0+\ell\|=1
    \quad\Longleftrightarrow\quad
    k\in\{9,12,13,16\}.
\end{equation}
For these four indices the unique shortest representatives are,
respectively, $a,b,-b,-a$. All other minima in
\eqref{eq:sharpness-initial-minima}, as well as the first lattice minimum
\[
    \lambda_1(L)=\min\{\|\ell\|:\ell\in L,\ \ell\ne0\},
\]
are strictly greater than one.

Only finitely many lattice cosets are involved, and each is discrete.
Hence the four unit representatives are separated by a positive gap from
all other representatives relevant to these minima. Since
$\varphi_a$ and $\varphi_b$ are linearly independent, choose
$h\in\mathbb R^2$ such that
$\varphi_a(h)=2$ and $\varphi_b(h)=1$, and put
$\alpha_\varepsilon=\alpha_0+\varepsilon h$. For $1\leq k\leq16$, write
\[
    d_k(\varepsilon)=\min_{\ell\in L}\|k\alpha_\varepsilon+\ell\|.
\]
For all sufficiently small $\varepsilon>0$, the same four representatives
remain uniquely shortest in their cosets. By differentiability at the smooth
points $a,b,-b,-a$ and by symmetry of the norm,
\[
    d_9(\varepsilon)=1+18\varepsilon+o(\varepsilon),\qquad
    d_{12}(\varepsilon)=1+12\varepsilon+o(\varepsilon),
\]
while
\[
    d_{13}(\varepsilon)=1-13\varepsilon+o(\varepsilon),\qquad
    d_{16}(\varepsilon)=1-32\varepsilon+o(\varepsilon).
\]
Consequently, after decreasing $\varepsilon$ if necessary,
\begin{equation}\label{eq:sharpness-critical-order}
    d_9(\varepsilon)>d_{12}(\varepsilon)>
    d_{13}(\varepsilon)>d_{16}(\varepsilon),
\end{equation}
whereas $d_k(\varepsilon)>d_9(\varepsilon)$ for every
noncritical $k\in\{1,\ldots,16\}$ and
$\lambda_1(L)>d_9(\varepsilon)$.

Fix such an $\varepsilon$, write $\alpha=\alpha_\varepsilon$ and
$d_k=d_k(\varepsilon)$, and set
\[
    D_j=\min\{\lambda_1(L),d_1,\ldots,d_j\}.
\]
For the first $17$ points of the Kronecker sequence, the possible nonzero
differences from the point with index $n$ have absolute indices
$1,\ldots,\max\{n-1,17-n\}$; choosing a different lift of the same torus
point contributes $\lambda_1(L)$. Therefore
\begin{equation}\label{eq:sharpness-delta-running-min}
    \delta_{n,17}^{\|\cdot\|}=D_{\max\{n-1,17-n\}} .
\end{equation}
As $n$ ranges from $1$ to $17$, the subscript on the right assumes every
value from $8$ to $16$. In this range the running minimum changes precisely
at $j=9,12,13,16$, by \eqref{eq:sharpness-critical-order} and the
inequalities for the noncritical indices. Hence the values in
\eqref{eq:sharpness-delta-running-min} consist of exactly the five distinct
numbers
\[
    D_8,\quad d_9,\quad d_{12},\quad d_{13},\quad d_{16}.
\]
Thus $g_{17}^{\|\cdot\|}(\alpha,L)=5$, completing the proof.

\section*{Acknowledgements}

We are grateful to Jens Marklof for suggesting the norm-perturbation
argument used to remove the strict convexity assumption, and to Alexey Glazyrin and Carl Dettmann for valuable comments. We also thank
Nikolay Moshchevitin for his interest in this work. Nikita Mironov is also grateful to the organisers of the Research Experience Program for Undergraduates 2026 (LIPS) for providing an excellent research platform and continuous support. The program was hosted by the Laboratory of Combinatorial and Geometric Structures at the Phystech School of Applied Mathematics and Informatics, MIPT.

\vspace{0.8cm}
N. A. Mironov, HSE University, dept. of Mathematics, 6 Usacheva Ulitsa, Moscow, 119048, Russia.

E-mail address: nikita.miron.444@gmail.com 
\vspace{0.5cm}

O. R. Musin, University of Texas Rio Grande Valley, School of Mathematical and Statistical Sciences, One West University Boulevard, Brownsville, TX, 78520, USA.

E-mail address: oleg.musin@utrgv.edu

\end{document}